\documentclass[11pt]{amsart}

\usepackage[colorlinks=true, pdfstartview=FitV, linkcolor=blue, 
citecolor=blue]{hyperref}

\usepackage{amssymb,amsmath,amscd}
\usepackage{graphicx}
\usepackage{bbm}
\usepackage{a4wide}
\usepackage{enumitem}
\usepackage{verbatim}
\usepackage{float}
\usepackage{accents}
\usepackage{stackrel}
\usepackage{faktor}
\usepackage{epsfig,psfrag}  % added psfrag
\usepackage{etex} %fuer gleichzeitigen Gebrauch fuer tikz und pstricks-add
\usepackage{pgfplots}
\usepackage{array}
\usepackage{mathrsfs}
\usepackage{amsthm}
\usepackage{mathtools}
\usepackage[compress,square,comma,authoryear]{natbib}
\setcitestyle{citesep={;},open={(},close={)},aysep={}}

\theoremstyle{plain}
\newtheorem{theorem}{Theorem}[section]
\newtheorem{prop}[theorem]{Proposition}
\newtheorem{lemma}[theorem]{Lemma}

\theoremstyle{definition}
\newtheorem{remark}[theorem]{Remark}

\newtheorem{definition}[theorem]{Definition}

\newcommand{\A}{\mathcal{A}}

\newcommand{\AS}{\mathbb{S}}
\newcommand{\ts}{\hspace{0.5pt}}
\newcommand{\nts}{\hspace{-0.5pt}}

\newcommand{\RR}{\mathbb{R}\ts}
\newcommand{\PP}{\mathbb{P}\ts}

\newcommand{\SSS}{\mathbb{S}}
\newcommand{\LL}{\mathbb{L}}

\newcommand{\NN}{\mathbb{N}}

\newcommand{\EE}{\mathbb{E}}
\newcommand{\cA}{\mathcal{A}}

\newcommand{\cF}{\mathcal{F}}
\newcommand{\cO}{\mathcal{O}}
\newcommand{\cP}{\mathcal{P}}
\newcommand{\cQ}{\mathcal{Q}}
\newcommand{\cR}{\mathcal{R}}

\newcommand{\one}{\mathbbm{1}}

\newcommand{\bd}{{\boldsymbol{d}}}

\newcommand{\bdelta}{{\boldsymbol{\delta}}}

\newcommand{\bepsilon}{{\boldsymbol{\varepsilon}}}
\newcommand{\blambda}{{\boldsymbol{\lambda}}}

\newcommand{\bSigma}{{\boldsymbol{\varSigma}}}

\newcommand{\bpmax}{{\boldsymbol{\pmax}}}
\newcommand{\bcdot}{{\boldsymbol{\cdot}}}

\newcommand{\bcQ}{{\boldsymbol{\cQ}}}
\newcommand{\Sigmacont}{\varSigma^{\text{c}}}
\newcommand{\bSigmacont}{\boldsymbol{\varSigma}^{\text{c}}}

\newcommand{\udo}[1]{\underaccent{$\text{.}$}{#1\ts}\nts}

\newcommand{\pmin}{\ts\ts\underline{\nts\nts 0\nts\nts}\ts\ts}
\newcommand{\pmax}{\ts\ts\underline{\nts\nts 1\nts\nts}\ts\ts}

\newcommand{\ee}{\mathrm{e}}
\newcommand{\dd}{\, \mathrm{d}}

\newcommand{\bs}{\boldsymbol}

\newcommand{\trans}{\raisebox{1pt}{$\scriptscriptstyle\mathsf{T}$}}

\newcommand{\myfrac}[2]{\frac{\raisebox{-2pt}{$#1$}}
      {\raisebox{0.5pt}{$#2$}}}

\definecolor{gre}{rgb}{.06,.49,0.03} % hexa 107d08

\DeclareMathOperator{\id}{Id}

\newcommand{\defeq}{\mathrel{\mathop:}=}
\newcommand{\eqdef}{=\mathrel{\mathop:}} 
\newcommand{\bigo}[1]{\mathcal{O}(#1)}

\begin{document}

\title[Ancestral lines under migration and recombination] {Solving the migration-recombination equation from a genealogical point of view}

\author{F. Alberti}
\address[F. Alberti]{Faculty of Mathematics,  Bielefeld University, \newline
\hspace*{\parindent}Postbox 100131, 33501 Bielefeld, Germany}
\email{falberti@math.uni-bielefeld.de}

\author{E. Baake}
\address[E. Baake]{Faculty of Technology,  Bielefeld University, \newline
\hspace*{\parindent}Postbox 100131, 33501 Bielefeld, Germany}
\email{ebaake@techfak.uni-bielefeld.de}

\author{I. Letter}
\address[I. Letter]{Statistics Department, University of Oxford, \newline
\hspace*{\parindent}24-29 St Giles, Oxford OX1 3LB, United Kingdom}
\email{restucci@stats.ox.ac.uk}

\author{S. Mart\'inez}
\address[S. Mart\'inez]{Departement of Mathematical Engineering and \newline
\hspace*{\parindent}Center of Mathematical Modeling, \newline
\hspace*{\parindent}UMI 2807 UCHILE-CNRS,  Universidad de Chile, \newline
\hspace*{\parindent}Santiago, CHILE}
\email{smartine@dim.uchile.cl}

\begin{abstract}
We consider the discrete-time migration-recombination equation, a deterministic, nonlinear dynamical system that describes the evolution of the genetic type distribution of a population evolving under migration and recombination in a law of large numbers setting.  We relate this dynamics (forward in time) to a Markov chain, namely a labelled partitioning process, backward in time. This way, we obtain a stochastic representation of the solution of the migration-recombination equation. As a consequence, one obtains an explicit solution of the nonlinear dynamics, simply in terms of powers of the transition matrix of the Markov chain. The limiting and quasi-limiting behaviour of the Markov chain are investigated, which gives immediate access to the asymptotic behaviour of the dynamical system. We finally sketch the analogous situation in continuous time.

\end{abstract}

\maketitle

\noindent \emph{keywords:} migration-recombination equation; ancestral recombination graph; duality; labelled partitioning process; quasi-stationarity; Haldane linearisation

\bigskip


\section{Introduction}
\label{Section0}
Recombination is a genetic mechanism that `mixes' or `reshuffles' the genetic
material of different individuals from generation to generation; it takes place during the reproductive cycle of sexually reproducing organisms. The analysis of models that describe the evolution of populations
under recombination  together with other processes are
among the major challenges in population genetics.

In this contribution, we consider 
the evolution under the joint action of recombination and migration of individuals between discrete locations (or \emph{demes}); we mainly focus on discrete time, where generations do not overlap. We will be concerned with a deterministic approach here, which assumes that the population is so large that a law of large numbers applies and random fluctuations (`genetic drift')  may be neglected. The resulting \emph{migration-recombination equation} is a large, nonlinear dynamical system
that describes the evolution of the genetic composition of each local
population over time, where the genetic composition
is identified with a probability distribution (or measure) on a space of
sequences of finite length. This model is a variant of the migration-selection-recombination equation formulated by~\citet{Buerger09}, who analysed its asymptotic behaviour in the classical dynamical systems setting, forward in time. It is our goal to complement this picture by relating this nonlinear dynamical
system to a linear one by embedding the solution into a higher dimensional space, a technique known as \emph{Haldane linearisation} \citep{HaleRingwood,Lyu} in the context of genetic algebras. This extends the approach taken by~\citet{haldane} to the case with migration. The resulting linear system has a natural interpretation as a Markov chain on the set of \emph{labelled} partitions of the set of sequence sites. Intuitively, this Markov chain describes how the genetic material of an individual from the current population is partitioned across an increasing number of ancestors, along with their locations, as the lines of descent are traced back into the past. This backward (or \emph{dual}) process combines a variant of the \emph{ancestral recombination graph}
\citetext{\citealp{hudson,griffithsmarjoram96,griffithsmarjoram97,BhaskarSong}; see also \citealp[Ch.~3.4]{durrett}} with a variant of the \emph{ancestral migration graph} \citep{Noto90,MatsenWakeley}. It is tractable in the law of large numbers regime considered here, due to the absence of coalescence; this was previously exploited for the recombination equation (without migration) by \citet{BaakeBaakeSalamat},  \citet{haldane} and  \citet{Martinez}; see \citet{recoreview} for a review. For an application of a similar idea in the context of the ancestral selection graph, see~\citet{SladeWakeley}.

All this leads to a stochastic representation of the solution of the (nonlinear, deterministic) migration-recombination equation in terms of the labelled partitioning process. As a consequence, one obtains an explicit solution of the nonlinear dynamics, simply in terms of powers of the transition matrix of the Markov chain. In particular, the asymptotic behaviour of the migration-recombination equation emerges without any additional effort, via the (unique) absorbing state of the Markov chain. We also investigate the quasi-limiting behaviour of the Markov chain, based on ideas by~\citet{Martinez}.

The paper is organised as follows. In Section~\ref{Section1}, we set the scene and introduce the forward-time model. In Section~\ref{sec:furthernotions}, we use the notion of (labelled) \emph{recombinators} to reformulate the forward model in a compact way. A crucial property of the dynamics, namely its consistency under marginalisation, is established in Section~\ref{sec:marginal}. The core of the paper is Section~\ref{sec:forwardsolution}, where we solve the forward iteration, together with Section~\ref{sec:stochastic}, which establishes the connection to the labelled partitioning process in terms of a duality, together with a genealogical interpretation. Section~\ref{sec:asymptotics} is devoted to its limiting and quasi-limiting behaviour, and Section~\ref{sec:continuoustime} sketches how the approach carries over to continuous time.

\section{The migration-recombination model}
\label{Section1}
Let us recapitulate the discrete-time \emph{migration-recombination equation} by~\citet{Buerger09}. 
The genetic information of an individual is
encoded in terms of a finite sequence of letters, indexed by the set
$[n] \defeq \{1,\ldots,n\}$
of sequence sites, where $n>1$ is the fixed length of the sequences. The sites
may either be interpreted as nucleotide positions in a DNA sequence, 
or as gene loci on
a chromosome.  For each site $i \in [n]$, there is a set (or \emph{alphabet}) $\A_i$
of \emph{letters} (to be interpreted as
nucleotides or alleles) that may possibly occur at that site.  For the sake of
simplicity, we restrict ourselves to \emph{finite} sets $\A_i$ here, but
this generalises easily.
A \emph{type} is thus identified with a sequence
\[
a=(a^{}_1, \dots ,
a^{}_n) \in \A_1 \times \cdots \times \A_n \eqdef \A\,, 
\]
where $\A$ is called the
\emph{type space}. 
We denote by $\cP(\cA)$ the set of all probability measures on $\cA$. We will also refer to such a probability measure as a  \emph{type distribution} or \emph{population}. This implies that we consider \emph{haploid individuals} or \emph{gametes};  it will be sufficient to work at this level, since, in contrast to \citet{Buerger09}, we do not consider selection.  Indeed, in the absence of selection, diploid genotypes are independent combinations of haploid gametes at all stages of the life cycle, that is, one has Hardy-Weinberg equilibrium throughout.

It will be crucial for the later analysis to not only consider complete  sequences (defined over the full set $[n]$), but also (sub)sequences (`marginal' types) that are defined over subsets of $[n]$. Given $U \subseteq [n]$, we set
\begin{equation*}
\cA_U \defeq \bigtimes_{i \in U} \cA_i.
\end{equation*}
Note that $\cA_{[n]}=\cA$.  Furthermore, $\cA_\varnothing$ is the empty Cartesian product, which is a set with a single element, namely the empty sequence $e$.
For $V \subseteq U \subseteq [n]$ and $a^{} \in \cA_U$, we define the corresponding \emph{marginal type} with respect to $V$ by
\begin{equation*}
a^V \defeq (a^{}_i)^{}_{i \in V}.
\end{equation*}
In line with this, for $\nu^{} \in \cP(\cA_U)$, we define its \emph{marginal distribution} (with respect to $V$) as the probability distribution on $\cA_V$ given by
\begin{equation*}
\nu^V (E) \defeq \nu^{}(E \times \cA_{U \setminus V})
\end{equation*}
for all $E \subseteq \cA_V$. In words, $\nu^V(E)$ is the probability that the marginal with respect to $V$ of a randomly sampled type from $\nu^{}$ agrees with some element of $E$.
In somewhat more technical terms, $\nu^V$ is the push-forward of $\nu^{}$ under the canonical projection from $\cA_U$ to $\cA_V$. Clearly, the map $\nu^{} \mapsto \nu^V$ is linear.
 
In order to discuss migration, we introduce a finite set $L$ of \emph{locations} (or \emph{demes}).  The  population at  location $\alpha \in L$ in generation $t \in \NN_0$ is denoted by $\mu_t(\alpha) \in \cP(\cA)$.
The collection of all local populations  is summarised into the (column) vector $\mu_t = \big (\mu_t(\alpha) \big )_{\alpha \in L} \in \cP(\cA)^L$; we call $\mu$ a \emph{spatially structured population}, or a \emph{metapopulation}. Moreover, for $U \subseteq [n]$, we write $\mu^U_t = \big (\mu^U_t(\alpha) \big )_{\alpha \in L}$ for the vector of marginal populations.
Throughout, indices of vectors and matrices that refer to locations are written as arguments.
Unless stated otherwise,  vectors are understood as column vectors.

The migration-recombination equation is a discrete-time dynamical system that describes the deterministic evolution of a metapopulation with non-overlapping generations. We assume that, in each generation, this evolution proceeds in two stages. First, individuals migrate between locations; then, random mating takes place among individuals at the same location, followed by reproduction involving recombination\footnote{While recombination, strictly speaking, does \emph{not} occur during reproduction itself, this is not relevant in the simplified setting of our model; simply put, as we are working at the level of gametes, the word `reproduction' refers, in this context, to the formation of new germ cells prior to mating.}. Discrete generations will be indexed by $t \in \NN_0$, where a population at time $t$ is understood as the population after the $t$-th round of recombination, but before migration; we will use the corresponding half integers $t+\frac{1}{2}$ to indicate the population after migration, but before recombination.

\subsection{Describing migration}
\label{subsec:descmig}
We first consider migration, following the presentation by \citet[Chap.~6.2]{Nagy}.
The most straightforward way to describe migration is via the so-called \emph{forward migration matrix} $\tilde{M}$. It is a stochastic matrix indexed by $L$, where the entry $\tilde{M}(\alpha,\beta)$ is the probability that a randomly chosen individual at location $\alpha$ migrates to location $\beta$ in the next generation.
However, it is more convenient to work instead with the \emph{backward migration matrix} $M$. It is also a stochastic matrix, and $M(\alpha,\beta)$ is the probability that a randomly chosen individual that currently lives at location $\alpha$ has migrated from location $\beta$. 
We assume that the local population sizes  $c(\alpha) \in \RR_{>0}$ remain constant over time. This is the case if either
\begin{equation}\label{constpopsizes}
c(\alpha) = \sum_{\beta \in L} c(\beta) \tilde{M}(\beta,\alpha) 
\end{equation}
for all $\alpha \in L$, or if population regulation takes place after the migration step.
In any case, denoting the location of a randomly sampled individual at time $t+ \frac{1}{2}$ by $\ell_{t+\frac{1}{2}}$ and its location in generation $t$ by $\ell_{t}$, we have 
\begin{equation*}
\begin{split}
M(\alpha,\beta) &= \PP \big (\ell_{t} = \beta \mid \ell_{t+\frac{1}{2}} = \alpha \big ) 
= \myfrac{\PP\big (\ell_{t} = \beta, \ell_{t+\frac{1}{2}} = \alpha \big )}{\PP \big (\ell_{t+\frac{1}{2}} = \alpha \big )}\\ 
&= \myfrac{\PP \big (\ell_{t+\frac{1}{2}} = \alpha \mid \ell_{t} = \beta \big ) \PP \big (\ell_{t} = \beta \big )}{\PP\big (\ell_{t+\frac{1}{2}} = \alpha \big )}
= \myfrac{c(\beta)}{c(\alpha)} \tilde{M}(\beta,\alpha).
\end{split}
\end{equation*}
Note that $M$ is stochastic by definition,  i.e. $\sum_{\beta \in L} M(\alpha, \beta)=1$ and $M(\alpha,\beta) \geqslant 0$ for all $\alpha,\beta \in L$.

For additional background, see \citet[Chapter 6.2]{Nagy}. 
In what follows, we will work exclusively with the backward migration matrix. Since we are only interested in relative type frequencies, the population sizes $c(\alpha)$ are irrelevant.
After migration (but before recombination), the local population at  $\alpha$ is therefore given by
\begin{equation}\label{distaftermigration}
\mu_{t + \frac{1}{2}}(\alpha) = \sum_{\beta \in L} M(\alpha, \beta) \mu_t(\beta),
\end{equation}
and the metapopulation may be written compactly as  
\begin{equation}\label{aftermigrationvector}
\mu_{t + \frac{1}{2}} = M \mu_t.
\end{equation}

\subsection{Describing recombination}
\label{subsec:descreco}
To describe recombination, we slightly modify the model by \citet{Buerger09} and follow the notation of \citet{Martinez}.
Here, the partitions of $[n]$ and its subsets will play a central role, see also \citet{haldane} or \citet{BaakeBaakeSalamat}. For $U \subseteq [n]$, a partition of $U$ is a set $\delta$ of mutually disjoint, non-empty subsets of $U$ whose union is $U$. We will also refer to the elements of a partition as \emph{blocks}. The set of all partitions of $U$ is denoted by $\AS(U)$. We say that $\varepsilon$ \emph{is finer than $($is a refinement of$)$} $\delta$  ($\varepsilon \preccurlyeq \delta$) if every block of $\varepsilon$ is contained in some block of $\delta$. The relation $\preccurlyeq$ defines a partial order on $\AS(U)$. We denote the unique minimal and maximal elements in $\AS(U)$ by
$
\pmin_U^{} \defeq \big  \{\{i\} : i \in U \big \}
$
and
$
\pmax_U^{} \defeq \{U\};
$
when $U = [n]$, we drop the subscript and write $\pmin$ and $\pmax$, rather than $\pmin_{[n]}$ and $\pmax_{[n]}$.
By
 $\delta \wedge \varepsilon$  we denote the \emph{coarsest common refinement} of $\delta$ and $\varepsilon$, that is,
\begin{equation*}
\delta \wedge \varepsilon \defeq \{d \cap e  : d \cap e \neq \varnothing, d \in \delta, e \in \varepsilon \};
\end{equation*}
it is the coarsest partition finer than both $\delta$ and $\varepsilon$.

We say that an offspring of a local population $\nu$ is \emph{recombined according to} $\delta = \{d_1,\ldots,d_m\} \in \AS([n])$  if it has $m$ parents of types $a^{(1)}, \ldots, a^{(m)} \in \cA^{[n]}$, all sampled independently from $\nu$, and inherits the letters at the sites in $d_i$ from the parent of type $a^{(i)}$. That is, the type of the offspring is $b = (b_1,\ldots,b_n)$, where 
$
b_i \defeq a^{(j)}_i $ if  $i \in d_j$.
The biologically reasonable cases are $m=1$ (then $\delta = \pmax$ and the full offspring sequence is inherited from a single parent) and $m=2$ (the offspring sequence is pieced together from two parents). The choice $m>2$ implies more than two parents, which is not biologically realistic, but we include this case (and, in this way, generalise \citet{Buerger09}) since it is mathematically interesting  and does not require additional effort. It is then clear that the type of an offspring of $\nu$ that is recombined according to $\delta$ has the distribution
\[
\bigotimes_{d \in \delta}\nu^d.
\]
That is, recombination according to $\delta$ turns $\nu$ into the product measure of the marginals with respect to the blocks of $\delta$; this reflects the random mating, that is, the independence of the parents. Again, we understand this product to respect the ordering of the sites.  

We assume that, in each time step, the entire local population is replaced; the proportion of individuals that are replaced by offspring recombined according to $\delta \in \AS([n])$ is denoted by $r_\delta^{} \geqslant 0$, where  
\begin{equation*}
\sum_{\delta \in \AS([n])} r_\delta^{} = 1.
\end{equation*}
The collection $(r_\delta^{})^{}_{\delta \in \AS({[n]})}$ is known as the \emph{recombination distribution}.
Thus, the components of $\mu_{t+1}$ are given by
\begin{equation}\label{recomig_ode}
\mu_{t+1}(\alpha) =  \sum_{\delta \in \AS({[n]})} r^{}_\delta  \bigotimes_{d \in \delta} \mu^d_{t + \frac{1}{2}}(\alpha) = \sum_{\delta \in \SSS({[n]})} r^{}_\delta  \bigotimes_{d \in \delta}  \sum_{\beta \in L} M(\alpha, \beta) \mu^d_t(\beta),
\end{equation}
where we have used \eqref{distaftermigration} and the linearity of marginalisation in the last step.
Equation~\eqref{recomig_ode} is called the \emph{migration-recombination equation}, or MRE for short. 

\section{Reformulation of the model}
\label{sec:furthernotions}
Extending concepts established by \citet{BaakeBaakeSalamat}, \citet{haldane} and \citet{Martinez}, we now formulate the MRE~\eqref{recomig_ode} in a more compact way.
This involves labelling the blocks of a partition by elements of $L$ to keep track of where the letters in the blocks come from. 
\begin{definition}\label{labelledpartitions}
A labelled partition of $U \subseteq {[n]}$ is a collection $\bdelta := \{\bd_1, \ldots, \bd_m\}$ for some $m \leqslant \lvert U \rvert$, where
$\bd_i = (d_i,\lambda_i)$,  $\delta = \{d_1, \ldots, d_m\}$ is a partition of $U$, and $\lambda_i \in L$ for $1 \leqslant i \leqslant m$. 
We call $\delta$ the \emph{base} of $\bdelta$, refer to its elements as the blocks of $\bdelta$, and interpret  $\lambda_i$  as the \emph{label} of block $d_i$. We write $\LL \AS (U)$ for the set of all labelled partitions of $U$. 
\hfill $\diamondsuit$
\end{definition} 

In order to rewrite Eq.~\eqref{recomig_ode}, we now introduce the \emph{labelled recombinator}. It is the labelled analogue of the recombinator used by
\cite{haldane} for unlabelled partitions. Since we will later also be interested in the evolution of the distribution of subsequences (compare Section~\ref{sec:marginal}), we introduce the concept in the required generality right away. 
\begin{definition}\label{labelledrecombinators}
Let $U \subseteq {[n]}$ and $\bdelta \in \LL \AS (U)$. Then, the labelled recombinator (with respect to $\bdelta$), namely $\cR^{U}_{\bdelta} : \cP(\cA_U^{})^L \to \cP(\cA_U^{})$, is defined by
\begin{equation*}
\cR_{\bdelta}^{U} (\nu) \defeq \bigotimes_{(d,\lambda) \in \bdelta} \nu^d  (\lambda ) \ts ;
\end{equation*}
if $U = {[n]}$, we will drop the superscript and write $\cR^{}_{\bdelta}$ instead of $\cR^{{[n]}}_{\bdelta}$.
\hfill $\diamondsuit$
\end{definition}

In words, $\cR_{\bdelta}^{}(\nu)$ is the distribution of the type of an offspring individual that is recombined according to $\delta$, where the parent of the labelled block $(d,\lambda)$ is sampled from the local population $\nu( \lambda )$. A similar interpretation holds for the \emph{marginal} recombinators; see Theorem~\ref{marginalisation} and Remark~\ref{onestepancestrylabelled}.
With this, Eq.~\eqref{recomig_ode} can be restated as follows.
\begin{lemma}\label{zusammenfassen}
The \textnormal{MRE}~\eqref{recomig_ode} can be written as
\begin{equation}\label{mainrewritten}
\mu_{t+1} = \sum_{\bdelta
\in \LL\SSS({[n]})} p^{}_{\bdelta} \cR_{\bdelta}^{} (\mu_t)
\end{equation}
with 
\begin{equation*}
p^{}_\bdelta \defeq \big (p^{}_\bdelta(\alpha) \big)_{\alpha \in L},
\end{equation*}
where the \emph{migration-recombination probabilities} are given by
\begin{equation*}
p^{}_{\bdelta}(\alpha) \defeq  r_\delta^{} \prod_{(d,\lambda) \in \bdelta} M  (\alpha,\lambda)
\end{equation*}
and are normalised, i.e. satisfy
\begin{equation*}
\sum_{\bdelta \in \LL \AS({[n]})} p^{}_\bdelta(\alpha) = 1 \text{ for all } \alpha \in L.
\end{equation*}
\end{lemma}
\begin{proof}
This follows immediately from Definition~\ref{labelledrecombinators} by expanding the measure product in Eq.~\eqref{recomig_ode}:
\begin{equation*}
\begin{split}
\mu_{t+1}(\alpha) & = \sum_{\delta \in \AS({[n]})} r_\delta^{} \bigotimes_{d \in \delta}  \sum_{\lambda \in L} M(\alpha,\lambda) \mu^d_t(\lambda) = \sum_{\delta \in \AS({[n]})} \sum_{\blambda \in L^\delta} r_\delta^{} \prod_{d \in \delta} M \big (\alpha, \lambda_d \big ) \bigotimes_{d \in \delta}  \mu_t^d (\lambda_d )\\[2mm]
& = \sum_{\bdelta \in \LL\AS({[n]})} p_\bdelta^{} (\alpha) \bigotimes_{(d,\lambda) \in \bdelta} \mu_t^d (\lambda) = \sum_{\bdelta \in \LL\AS({[n]})} p_\bdelta^{} (\alpha) \cR_\bdelta^{}(\mu_t),
\end{split}
\end{equation*}
where, in the third step, we identified the double sum over all partitions of ${[n]}$ and all possible vectors of labels of their blocks with the sum over all labelled partitions. 
The normalisation is a consequence of $\sum_{\delta \in \AS({[n]})} r_\delta^{} = 1 =\sum_{\beta \in L} M(\alpha,\beta) $.
\end{proof}

We call the probability distribution $p(\alpha) = \big ( p_\bdelta(\alpha) \big )_{\bdelta \in \LL \SSS ([n])}$ the \emph{migration-recombination distribution at $\alpha$}. 
\begin{remark}\label{babyduality}
Lemma~\ref{zusammenfassen} has a simple stochastic interpretation. To sample the type of an individual in generation $t+1$ (say, at location $\alpha$), we first pick a random labelled partition $\bdelta$ according to $p^{}(\alpha)$ and subsequently sample from $\cR^{}_\bdelta(\mu_t)$. 
The factorisation of $p_\bdelta^{}(\alpha)$ in Lemma~\ref{zusammenfassen} implies that the genome is first partitioned across its parents according to $\delta$, with probability $r_\delta$. Subsequently, the label is reassigned (conditionally) independently for each block, according to $M(\alpha,\bcdot)$, as we trace back the origin of each ancestor. Finally, the offspring type is determined by piecing together (fragments of) independent samples of the ancestral sequences at the appropriate locations, in generation $t$. This leads to the product measure in Definition~\ref{labelledrecombinators}. 
We will further elaborate on this in Section~\ref{sec:stochastic}.   
\hfill $\diamondsuit$
\end{remark}

To continue, we need a few additional concepts around labelled partitions.
First, the notion of an \emph{induced \textnormal{(}labelled\textnormal{)} partition} is required. For $\varnothing \neq V \subseteq U$ and $\bdelta \in \LL \AS (U)$, we denote by $\bdelta|_{V}$ the labelled partition induced by $\bdelta$ on $V$; it is given by
\[
  \bdelta|_{V} \defeq \{ (d \cap V, \lambda) : d \cap V \neq \varnothing, (d,\lambda) \in \bdelta\}
\]
with base
\begin{equation*}
\delta|_{V} = \{ d \cap V : \varnothing \neq d \cap V,\, d \in \delta  \},
\end{equation*}
the partition induced by the (unlabelled) partition $\delta$ on $V$. Simply put, every block inherits the label of the unique block of the original partition that contains it. 

Conversely, given a partition $\delta$ of $U$ and a family $(\bepsilon_d)_{d \in \delta}$ of labelled partitions of its blocks, their union
\begin{equation*}
\bigcup_{d \in \delta} \bepsilon_d
\end{equation*}
is a labelled partition of $U$; its base is the union
\begin{equation*}
\bigcup_{d \in \delta} \varepsilon_d
\end{equation*}
of the bases $\varepsilon_d$.

Finally, given two labelled partitions $\bdelta$ and $\bepsilon$, we say that $\bepsilon$ \emph{is finer than} $\bdelta$ ($\bepsilon \preccurlyeq \bdelta$) if $\varepsilon \preccurlyeq \delta$. The  partial order on $\AS(U)$ thus carries over to a partial order on $\LL \AS (U)$. For any $\alpha \in L$, there is a unique maximal element; namely, the labelled partition $\bpmax_U^\alpha \defeq \{({[n]},\alpha) \}$ that consists of a single block with label $\alpha$. If $U = {[n]}$, we drop the subscript.

\begin{remark}\label{refinementbijection}
It is not difficult to see that $\bepsilon \preccurlyeq \bdelta$ if and only if
\begin{equation*}
\bepsilon = \bigcup_{d \in \delta} \bepsilon|_{d}.
\end{equation*}
For a fixed $\delta \in \AS({[n]})$, this implies the following bijection between the labelled partitions $\bepsilon$ with $\varepsilon \preccurlyeq \delta$ and collections $(\bepsilon_d)_{d \in \delta}$ of labelled partitions of the individual blocks of $\delta$.
Given $\bepsilon$ with $\varepsilon \preccurlyeq \delta$, we obtain the collection $( \bepsilon|_d)_{d \in \delta}$ of labelled partitions induced by $\bepsilon$ on the blocks of $\delta$. Conversely, given a collection $(\bepsilon_d )_{d \in \delta}$ of labelled partitions of the blocks of $\delta$, we set $\bepsilon \defeq \bigcup_{d \in \delta} \bepsilon_d$; note that $\bepsilon \preccurlyeq \bdelta$ and $\bepsilon|_d = \bepsilon_d$. See also Fig.~\ref{fig:inducedjoin}.
\hfill $\diamondsuit$
\end{remark}

\begin{figure}[t]
\psfrag{1}{$1$}
\psfrag{2}{$2$}
\psfrag{3}{$3$}
\psfrag{4}{$4$}
\psfrag{5}{$5$}

\psfrag{d1}{$d_1$}
\psfrag{d2}{$d_2$}
\psfrag{d3}{$d_3$}
\psfrag{d4}{$d_4$}
\psfrag{d5}{$d_5$}
\psfrag{d6}{$d_6$}

\psfrag{ed1}{$\bepsilon^{}_{d_1} = \bepsilon|_{d_1}^{}$}
\psfrag{ed2}{$\bepsilon^{}_{d_2} = \bepsilon|_{d_2}^{}$}
\psfrag{ed3}{$\bepsilon^{}_{d_3} = \bepsilon|_{d_3}^{}$}
\psfrag{ed4}{$\bepsilon^{}_{d_4} = \bepsilon|_{d_4}^{}$}
\psfrag{ed5}{$\bepsilon^{}_{d_5} = \bepsilon|_{d_5}^{}$}
\psfrag{ed6}{$\bepsilon^{}_{d_6} = \bepsilon|_{d_6}^{}$}

\psfrag{delta}{$\delta$}
\psfrag{bepsilon}{$\bepsilon = \bigcup_{j = 1}^6 \bepsilon^{}_{d_j}$}
\includegraphics[width = 0.9\textwidth]{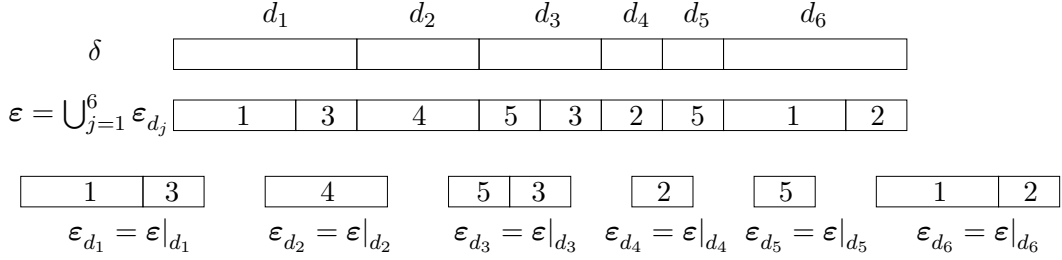}
\caption{\label{fig:inducedjoin}
At the top, an unlabelled partition of $[n]$. In the middle, a labelled refinement of $\delta$, which gives rise to labelled partitions of the blocks of $\delta$ (bottom). Conversely, one can start with the collection of labelled partitions at the bottom and join them to obtain a labelled refinement of $\delta$.
}
\end{figure}

We will now see that the recombinator for a union of labelled partitions of disjoint subsets is the product of the recombinators for the individual labelled partitions; compare also \citet[Lemma 2]{haldane}.
\begin{lemma}\label{recoindcot}
Let $\delta \in \AS({[n]})$ and $\bepsilon_d \in \LL \AS (d)$ for all $d \in \delta$. Then, for all $\nu \in \cP(X)^L$,
\begin{equation*}
\cR_{\bigcup_{d \in \delta} \bepsilon_d }^{} (\nu) = \bigotimes_{d \in \delta} \cR_{\bepsilon_d}^{d}(\nu^{d}).
\end{equation*}
In particular, for $\bepsilon \in \LL \AS({[n]})$ with $\varepsilon \preccurlyeq \delta$, we have 
\begin{equation*}
\cR_{\bepsilon}^{} (\nu) = \bigotimes_{d \in \delta} \cR_{\bepsilon|_d}^{d}(\nu^{d}).
\end{equation*}
\end{lemma}
\begin{proof}
For the first claim, we write out the labelled recombinators and see that
\begin{equation*}
\bigotimes_{d \in \delta} \cR_{\bepsilon_d}^d (\nu^d) = \bigotimes_{d \in \delta} \bigotimes_{(e,\lambda) \in \bepsilon_d} \nu^e  ( \lambda  ) = 
\bigotimes_{(e,\lambda) \in \bigcup_{d \in \delta} \bepsilon_d} \nu^e  ( \lambda ) = \cR_{\bigcup_{d \in \delta} \bepsilon_d }^{}(\nu).
\end{equation*}
For the second claim, see Remark~\ref{refinementbijection}.
\end{proof}

We now turn to the marginalisation consistency of the MRE~\eqref{recomig_ode}, a property that will turn out as the key to its solution. 

\section{Marginalisation consistency}
\label{sec:marginal}
Just as in the continuous-time case for pure recombination treated by~\citet{haldane}, the marginalisation consistency of the model is a crucial ingredient. We therefore now turn to  the dynamics that the MRE~\eqref{recomig_ode} induces on the marginal type distributions. 
As a warm-up, we prove the following elementary, but useful, result. 
\begin{lemma}\label{productmarginalisation}
Let $U,V \subseteq {[n]}$, $U \cap V = \varnothing$, and let $\nu_U^{}, \nu_V^{}$ be probability measures on $\cA_U$ and $\cA_V$, respectively. Then, we have for any $W \subseteq U \cup V$
\begin{equation*}
(\nu_U^{} \otimes \nu_V^{})^W = \nu_U^{U \cap W} \! \otimes \nu_V^{V \cap W}.
\end{equation*}
\end{lemma}
\begin{proof}
Note that $\cA_W = \cA_{U \cap W} \times \cA_{V \cap W}$. Let us fix $E_{U \cap W} \subseteq \cA_{U \cap W}$ and $E_{V \cap W} \subseteq \cA_{V \cap W}$. Then, for any $W \subseteq U \cup V$,
\begin{equation*}
\begin{split}
(\nu_U^{} \otimes \nu_V^{})^W (E_{U \cap W}^{} \times E_{V \cap W}^{})  &= 
(\nu_{U}^{} \otimes \nu_V^{}) (E_{U \cap W}^{} \times E_{V \cap W}^{} \times \cA_{(U \cup V) \setminus W}^{}) \\&=
(\nu_U^{} \otimes \nu_V^{}) \big ( (E_{U \cap W}^{} \times \cA_{U \setminus W}^{}) \times (E_{V \cap W}^{} \times \cA_{V \setminus W}^{} )\big ) \\&=
\nu_U^{}(E_{U \cap W}^{} \times \cA_{U \setminus W}^{}) \nu_V^{} (E_{V \cap W}^{} \times \cA_{V \setminus W}^{}) \\&=
\nu_U^{U \cap W} (E_{U \cap W}^{}) \nu_V^{V \cap W} (E_{V \cap W}^{}).
\end{split}
\end{equation*}
\end{proof}

\begin{remark}\label{emptyprojection}
It is important to note that Lemma~\ref{productmarginalisation} remains true if $U \cap W = \varnothing$ or $V \cap W = \varnothing$. Assume, for instance, that $U \cap W = \varnothing$. Recalling that the empty Cartesian product $\cA_\varnothing^{}$ is the singleton $\{e\}$ (recall that $e$ is the empty sequence), $\nu_U^{U \cap W}$ is the unique probability measure on $\{e\}$ and can be treated as the scalar 1, in the sense that
\begin{equation*}
\nu_U^{U \cap W} \otimes \nu_V^{V \cap W} = \nu_V^{V \cap W} \otimes \nu_U^{U \cap W} =  \nu_V^{V \cap W}.
\end{equation*}
\hfill $\diamondsuit$
\end{remark}

We now prove the main result of this section, which shows that the MRE is consistent  under marginalisation. 
\begin{theorem}\label{marginalisation}
Let $(\mu_t)_{t \in \NN_0}^{}$ be a solution of the \textnormal{MRE}~\eqref{recomig_ode} and $U$ a nonempty subset of ${[n]}$. Then, $(\mu_t^U)_{t \in \NN_0}^{}$ satisfies the \emph{marginal MRE}
\begin{equation*}
\mu_{t+1}^U =  \sum_{\bdelta \in \LL\SSS(U)} p^{U}_{\bdelta} \cR^U_{\bdelta} (\mu^U_t),
\end{equation*}
where $p^U_\bdelta$%=\big (p^U_\bdelta (\alpha)\big)^{}_{\alpha \in L}$
is given by
\begin{equation*}
p^U_{\bdelta} \defeq \sum_{\substack{\bdelta' \in \LL \AS({[n]}) \\ \bdelta'|_{U} = \bdelta}} p_{\bdelta'}^{} \text{ for } \bdelta \in \LL\SSS(U).
\end{equation*}
\end{theorem}

\begin{proof}
By Lemma~\ref{zusammenfassen} and the linearity of marginalisation, we have 
\begin{equation*}
\mu_{t+1}^U = \bigg (\sum_{\bdelta' \in \LL\AS({[n]})} p_{\bdelta'}^{} \cR_{\bdelta'} (\mu_t) \bigg )^U = \sum_{\bdelta' \in \LL\AS({[n]})} p_{\bdelta'}^{} \big ( \cR_{\bdelta'}(\mu_t) \big )^U. 
\end{equation*}
Using Lemma~\ref{productmarginalisation}, we obtain for all $\bdelta' \in \LL \AS ({[n]})$
\begin{equation*}
\big ( \cR_{\bdelta'}(\mu_t) \big )^U = \bigg ( \bigotimes_{(d',\lambda') \in   \bdelta'} \mu_t^{d'}  (\lambda') \bigg )^U 
= \bigotimes_{\substack{(d',\lambda') \in \bdelta' \\ d' \cap U \neq \varnothing}} \mu_t^{d' \cap U} (\lambda' ) = 
 \bigotimes_{(d,\lambda) \in \bdelta'|_{U}} \mu_t^d (\lambda ) = \cR^U_{\bdelta'|_U^{}} (\mu_t^U),
\end{equation*}
where, in the second step, we ignored  the factors corresponding to $d'$ with $d' \cap U = \varnothing$ (compare Remark~\ref{emptyprojection}).
Thus,
\begin{equation*}
\mu_{t+1}^U = \sum_{\bdelta' \in \LL \AS({[n]})} p_{\bdelta'}^{} \cR_{\bdelta'|_U}^U(\mu_t^U) = \sum_{\bdelta \in \LL \AS(U)} p_{\bdelta}^U \cR^U_{\bdelta}(\mu_t^U),
\end{equation*}
which is what we wanted to show.
\end{proof}
The $p^U_\bdelta (\alpha)$  are the \emph{marginal} migration-recombination probabilities (at $\alpha$), and, accordingly, $p^U (\alpha) = \big (p^U_\bdelta(\alpha) \big )_{\bdelta \in \LL \SSS([n])}$ is called the marginal migration-recombination \emph{distribution} (at $\alpha$). We will now see that the marginal migration-recombination probabilities have a product structure analogous to that of the migration-recombination probabilities in Lemma~\ref{zusammenfassen}.

\begin{lemma}\label{factorisation_marginal}
The marginal migration-recombination probabilities $p_\bdelta^{U}(\alpha)$ from Theorem~\textnormal{\ref{marginalisation}} can be written as
\begin{equation*}
p^{U}_{\bdelta}(\alpha) = \Bigg (\sum_{\substack{\delta' \in \AS({[n]}) \\ \delta'|_U = \delta}} r_{\delta'}^{} \Bigg) \prod_{(d,\lambda) \in \bdelta} M  ({\alpha, \lambda}).
\end{equation*}
\end{lemma}
\begin{proof}
We write the (given) labelled partition $\bdelta$ as
\begin{equation*}
\bdelta = \{(d_1,\lambda_1),\ldots,(d_k,\lambda_k) \}.
\end{equation*}
Next, we split the conditional sum over the labelled partitions into the sums over the appropriate partitions and their labels. Thus,
\begin{equation}\label{eafe}
p_\bdelta^U(\alpha) = \sum_{\substack{\bdelta' \in \LL \SSS({[n]}) \\ \bdelta'|_{U} = \bdelta}} p_{\bdelta'}^{}(\alpha)  =
\sum_{\substack{ \delta' =  \{d_1',\ldots,d_m' \} \in \AS({[n]}) \\ \{d_1',\ldots,d_m' \}|_U = \delta}} r^{}_{\delta'} \sum_{\lambda_1',\ldots,\lambda_m' \in L} \prod_{j = 1}^k \one_{\lambda_j' = \lambda_j} \prod_{j = 1}^m M(\alpha,\lambda_j'),
\end{equation}
where the blocks are indexed so that $d_j' \cap U = d_j$ for all $1 \leqslant j \leqslant k$ and $d_j' \cap U = \varnothing$ for $k + 1 \leqslant j \leqslant m$. Clearly, 
\begin{equation*}
\begin{split}
\sum_{\lambda_1',\ldots,\lambda_m' \in L} \prod_{j = 1}^k \one_{\lambda_j' = \lambda_j} & \prod_{j = 1}^m M(\alpha,\lambda_j') \\ &=
\bigg ( \sum_{\lambda_1', \ldots, \lambda_k'}  \prod_{j = 1}^k \one_{\lambda_j' = \lambda_j} \prod_{j = 1}^k M(\alpha,\lambda_j') \bigg ) \bigg (\sum_{\lambda_{k + 1}', \ldots, \lambda_{m}'}  \prod_{j = k+1}^m M(\alpha,\lambda_j') \bigg ),
\end{split}
\end{equation*}
with the usual convention that the empty product is 1.
Now, we can use the indicator in the first bracket to eliminate the summation, yielding
\begin{equation*}
 \sum_{\lambda_1', \ldots, \lambda_k'}  \prod_{j = 1}^k \one_{\lambda_j' = \lambda_j} \prod_{j = 1}^k M(\alpha,\lambda_j')  = \prod_{j = 1}^k M(\alpha,\lambda_j).
\end{equation*}
The second bracket is equal to one, by the stochasticity of $M$:
\begin{equation*}
\sum_{\lambda_{k + 1}', \ldots, \lambda_{m}'}  \prod_{j = k+1}^m M(\alpha,\lambda_j') = \prod_{j = k + 1}^m \sum_{\lambda' \in L} M(\alpha,\lambda')  = 1.
\end{equation*}
Inserting this back into \eqref{eafe} finishes the proof.
\end{proof}

\begin{remark}\label{onestepancestrylabelled}
The same stochastic interpretation as for Eq.~\eqref{mainrewritten} (see Remark~\ref{babyduality}) holds also for the marginalised system.  With probability 
\begin{equation}\label{marginalrecombinationprobabilities}
r^U_\delta \defeq \sum_{\substack{\delta' \in \AS({[n]}) \\ \delta'|_U = \delta}} r_\delta^{},
\end{equation}
the subsequence with respect to $U$ of a sampled individual is partitioned across its ancestors according to $\delta$. Then, the labels are reassigned independently according to $M$, reflecting their independent migration.  
\hfill $\diamondsuit$
\end{remark}

\section{Solution of the forward iteration}
\label{sec:forwardsolution}

Next, we use the marginalisation consistency 
established in the previous section to tame the MRE~\eqref{recomig_ode}. As discussed by \citet{haldane} for pure recombination, 
the main idea is to consider the time evolution of the (column) vector 
$\cR^{} (\mu_t) = \big (\cR^{}_\bdelta(\mu_t) \big )_{\bdelta \in \LL \SSS([n])}$, 
rather than  $\mu_{t}$ alone; note that we recover $\mu_t(\alpha)$ as the $\bpmax^\alpha$-component of $\cR^{} (\mu_t)$.

\begin{theorem}
\label{linearisation}
Let $T$ be the matrix, indexed by the labelled partitions $\LL \SSS({[n]})$, with entries
\begin{equation*}
T_{\bdelta \bepsilon} = \left \{ \begin{array}{ll}
0, & \text{if } \bepsilon \not \preccurlyeq \bdelta, \\
\prod_{(d,\lambda) \in \bdelta} p_{\bepsilon_{|d}}^d  (\lambda), & \text{if } 
\bepsilon \preccurlyeq \bdelta,
\end{array} \right . 
\end{equation*}
where the $p_{\bepsilon_{|d}}^d  (\lambda)$ are as in Lemma~\textnormal{\ref{factorisation_marginal}}.
Then, $T$ is a stochastic matrix. Assume that $(\mu_t)_{t \in \NN_0}^{}$ satisfies the \textnormal{MRE}~\eqref{recomig_ode}. Then, $\cR^{}(\mu_t)$ satisfies the linear recursion
\begin{equation*}
\cR^{}(\mu_{t+1}) =  T \cR^{}(\mu_t).
\end{equation*}
In particular,
\begin{equation*}
\cR^{}(\mu_t) =  T^t \cR^{}(\mu_0),
\end{equation*}
where $T^t$ denotes the $t$-th power of $T$.
\end{theorem}

\begin{proof}
By Definition~\ref{labelledrecombinators} and Theorem~\ref{marginalisation},
\begin{equation*}
\cR_\bdelta (\mu_{t+1}) = \bigotimes_{(d,\lambda) \in \bdelta} \sum_{\bepsilon_d \in \LL\SSS(d)} p_{\bepsilon_d}^d (\lambda ) \cR_{\bepsilon_d}^d (\mu_t^d) 
= \sum_{\substack{\bepsilon_d \in \LL \AS (d) \\ \forall d \in \delta}} \Big ( \prod_{(d,\lambda) \in \bdelta} p_{\bepsilon_d}^d  (\lambda) \Big ) \bigotimes_{d \in \delta} \cR_{\bepsilon_d}^d (\mu_t^d).
\end{equation*}
By Remark~\ref{refinementbijection}, the right-hand side is equal to
\begin{equation*}
\sum_{\udo{\bepsilon}  \preccurlyeq \bdelta} \Big ( \prod_{(d,\lambda) \in \bdelta} p_{\bepsilon_{|d}}^d  ( \lambda ) \Big )  \bigotimes_{d \in \delta} \cR_{\bepsilon|_{d}}^d (\mu_t^d) 
= \sum_{\udo{\bepsilon}  \preccurlyeq \bdelta} T_{\bdelta \bepsilon} \cR^{}_{\bepsilon} (\mu_t),
\end{equation*}
where we used Lemma~\ref{recoindcot} and the underdot indicates the summation variable. 
That $T$ is a stochastic matrix is a straightforward consequence of $p^d(\alpha)$ being a probability distribution on $\LL \AS (d)$ for all $d \subseteq {[n]}$ and all $\alpha \in L$.

%\begin{equation*}
%\sum_{\substack{\bepsilon \\ \varepsilon \preccurlyeq \delta}} \bT_{\bdelta \bepsilon} = \sum_{\substack{\bepsilon \\ \varepsilon \preccurlyeq \delta}} %\prod_{d \in \bdelta} r^d_{\bepsilon^{}_{|d}}\big ( \ell(\bdelta,d) \big ) = \prod_{d \in \bdelta} \sum_{\bepsilon_d \in \LL \SSS(d)} r^d_{\bepsilon^{}_d}%\big( \ell(\bdelta,d) \big ) = 1,
%\end{equation*}
\end{proof}

We have just witnessed how the solution of a nonlinear system, embedded in a higher dimensional space, turns into the solution of a linear system and may thus be given explicitly, simply via matrix powers. This is an extension of a technique called \emph{Haldane linearisation} \citep{HaleRingwood,haldane,recoreview} to the case with migration. The underlying mechanism can be found in the genealogical structure, which is discussed next.

\section{Stochastic Interpretation, genealogical content, and duality}

\label{sec:stochastic}
Let us now turn to the probabilistic content of Theorem \ref{linearisation}. We will see that the appearance of the stochastic matrix $T$ is no coincidence; rather, it has a natural interpretation as the transition matrix of a Markov chain describing the random genealogy of a single individual.
\begin{definition}\label{lfragprocessdef}
The \emph{labelled partitioning process} (LPP) is a discrete-time Markov chain $\bSigma \defeq \big (\bSigma_t \big )_{t \in \NN_0}$ 
with values in $\LL\SSS({[n]})$ and transition matrix $T$, that is,
\[
\PP(\bSigma_{t+1} = \bepsilon \mid \bSigma_t = \bdelta) = T_{\bdelta \bepsilon} 
\]
for all $\bdelta,\bepsilon \in \LL \AS ({[n]})$. 
\hfill $\diamondsuit$
\end{definition}

In words, $\bSigma_{t+1}$ is constructed from $\bSigma_t$ by independently replacing each labelled block $(d,\lambda) \in \bSigma_t$, with probability $p^d_{\bepsilon_d} (\lambda)$, by the (labelled) blocks of $\bepsilon_d$; see also Fig.~\ref{LPPillustrated}.

The genealogical interpretation of $\bSigma$, started in $\bpmax^{\alpha}$, is as follows. Each labelled block $(d,\lambda)$ of $\bSigma_t$ corresponds to a different ancestor of the individual at present, sampled at location $\alpha$, who lived at location $\lambda$, $t$ generations before the present.
The elements of $d$ are the sequence sites that are inherited from this ancestor. As we look one generation further into the past, $d$ is replaced by the blocks of a labelled partition $\bepsilon_d \in \LL \AS(d)$, which describes how the type of that ancestor is, in turn, pieced together from its parents, alive $t+1$ generations before the present. Note that, now, the labelled partitions of $d$ are relevant rather than those of ${[n]}$. This is because we already know that this ancestor only contributes sites contained in $d$, whence we only need to trace back the ancestry of these sites. (This reflects the marginalisation consistency of the model, compare Remark~\ref{onestepancestrylabelled}).
Furthermore, the various blocks  split independently as the population, in the law of large numbers regime  assumed here, is so large that two given individuals never share a common ancestor; thus, their lineages are conditionally independent. 

The connection between the solution of the MRE~\eqref{recomig_ode} and the genealogical process is formalised in the following theorem, which is a probabilistic restatement of Theorem~\ref{linearisation} and draws on the notion of duality for Markov processes~\citep{liggett,kurtjansen}; in particular, we think about the solution of the forward-time equation as a Markov chain with deterministic transitions.

\begin{theorem}\label{duality}
The \textnormal{LPP} and the solution of the \textnormal{MRE}~\eqref{recomig_ode} are dual with respect to the duality function
\begin{equation*}
(\bdelta,\mu) \mapsto \cR_{\bdelta}^{} (\mu).
\end{equation*}
That is, for all $\bdelta \in \LL\SSS({[n]})$ and all $\mu_0 \in \cP(\cA)^L$, we have
\begin{equation*}
\EE [ \cR_{\bSigma_t}^{}( \mu_0) \mid \bSigma_0 = \bdelta ] = \cR_\bdelta^{}(\mu_t).
\end{equation*}
In particular, this entails the stochastic representation
\begin{equation*}
\mu_t(\alpha) = \EE [ \cR_{\bSigma_t}(\mu_0) \mid \bSigma_0 = \bpmax^\alpha]
\end{equation*}
for the solution of the \textnormal{MRE}~\eqref{recomig_ode}.
\end{theorem}

\begin{proof}
We prove the theorem by induction over $t$. For $t = 0$, there is nothing to show. Assuming now that
\begin{equation*}
\EE [ \cR_{\bSigma_t}^{}( \mu_0) \mid \bSigma_0 = \bdelta ] = \cR^{}_\bdelta (\mu_t)
\end{equation*}
for any $t>0$,
we compute, using Theorem~\ref{linearisation} in the first step, the induction hypothesis in the second, time-homogeneity in the third, and the Markov property in the last:
\begin{equation*}
\begin{split}
\cR^{}_\bdelta(\mu_{t+1}) &= \sum_{\udo{\bepsilon} \preccurlyeq \bdelta} T_{\bdelta \bepsilon} \cR_\bepsilon(\mu_t) = \sum_{\udo{\bepsilon} \preccurlyeq \bdelta}  \PP[\bSigma_1 = \bepsilon \mid \bSigma_0 = \bdelta] \, \EE [ \cR_{\bSigma_t}^{}( \mu_0) \mid \bSigma_0 = \bepsilon]\\
&= \sum_{\udo{\bepsilon} \preccurlyeq \bdelta} \PP[\bSigma_1 = \bepsilon \mid \bSigma_0 = \bdelta] \, \EE [ \cR_{\bSigma_{t+1}}^{}( \mu_0) \mid \bSigma_1 = \bepsilon]  = \sum_{\udo{\bepsilon} \preccurlyeq \bdelta}\EE [ \cR_{\bSigma_{t+1}}^{}( \mu_0) \mid \bSigma_0 = \bdelta].
\end{split}
\end{equation*}
This proves the statement for $t + 1$. 
\end{proof}

Note that the duality function used here is vector valued. This is a slight extension of the standard notion, since the duality function is usually assumed to take values in $\RR$; see the references above for a thorough exposition. 

To get a better feel for this probabilistic way of thinking, we take advantage of the stochastic representation from Theorem~\ref{duality} to construct an explicit solution formula in the case of two sites. When evaluating the expectation, we distinguish two cases. Either, the two sites have not been separated until generation $t$, which happens with probability $r_{\pmax}^{t}$. In this case, both sites have the same ancestor who comes, with probability 
$(M^t)(\alpha ,\beta)$, from location $\beta$. Hence, in this case, $\mu_t^{}(\alpha) = (M^t \mu_0^{})(\alpha)$. If, on the other hand, the sites \emph{have} been separated, we denote by $\sigma$ the smallest $t$ such that $|\bSigma_t| = 2$. In this case, the letters come from two different parents. Their origins are determined by performing independent random walks on $L$ for the remaining time $t - \sigma + 1$. Summing over all possible values for $\sigma$ and the label of the block at the time of splitting (which is $\gamma$ with probability $(M^{\sigma -1})(\alpha, \gamma)$), we see that
\begin{equation} \label{discreteexample}
\mu_t^{}(\alpha) = r_{\pmax}^t (M^t \mu_0^{})(\alpha) + \sum_{\gamma \in L}  \sum_{\sigma = 1}^{t} r_{\pmax}^{\sigma-1} r_{\pmin}^{} (M^{\sigma - 1})(\alpha, \gamma) (M^{t - \sigma + 1} \mu^{}_0)^{\{1\}}(\gamma) \otimes (M^{t - \sigma + 1} \mu^{}_0)^{\{2\}}(\gamma).
\end{equation}

\begin{figure}[t]
\psfrag{1}{$1$}
\psfrag{2}{$2$}
\psfrag{3}{$3$}
\psfrag{4}{$4$}

\psfrag{t0}{$t = 0$}
\psfrag{t1}{$t = 1$}
\psfrag{t2}{$t = 2$}

\psfrag{3 to 2}{$3 \to 2$}
\psfrag{1 to 3}{$1 \to 3$}
\psfrag{1 to 2}{$1 \to 2$}
\psfrag{2 to 4}{$2 \to 4$}
\psfrag{2 to 1}{$2 \to 1$}

\includegraphics[width = 0.7\textwidth]{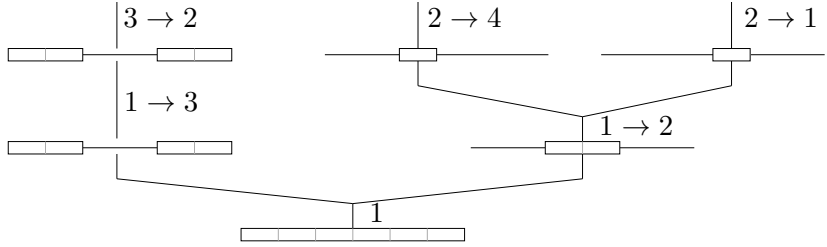}
\caption{\label{LPPillustrated}
An illustration of the LPP starting from $\bs{\pmax}^1$, the trivial partition consisting of a single block with label $1$; the set of locations is $L = \{1,2,3,4\}$. Backward time runs from bottom to top. In each generation, the blocks of the partition are first subject to individual splitting and we trace back the ancestral lines that belong to each fragment; compare Remarks~\ref{onestepancestrylabelled} and \ref{babyduality}. The fragments provided by each ancestor are labelled with their locations; we write $\alpha \to \beta$ to indicate migration from $\alpha$ to $\beta$. Recall that in the forward-time model, recombination occurs after migration. Thus, when looking backward in time, splitting due to recombination occurs before the reassignment of the label due to migration. In particular, the first event in this example is a splitting of our sequence located in deme 1.
}
\end{figure}

In the case without migration (i.e, when ignoring the labels), this genealogical process is a variant of the \emph{ancestral recombination graph} \citep{hudson,griffithsmarjoram96,griffithsmarjoram97,BhaskarSong}, which was used by \citet{haldane} to solve the recombination equation; see also \citet{recoreview}. More precisely, the \emph{unlabelled} partitioning process $\varSigma$ is simply the base of $\bSigma$. Likewise, the
transition matrix $T^{\! \text{ul}}$ of $\varSigma$ is obtained from $T$  by marginalising over the labels. Thus, $T^{\! \text{ul}}$ has the entries
\begin{equation}\label{Tulentries}
T^{\text{ul}}_{\delta \varepsilon} = \begin{cases} 
0, & \text{if } \varepsilon \not \preccurlyeq \delta, \\
\prod_{d \in \delta} r_{\varepsilon|_d}^d
, & \text{if } 
\varepsilon  \preccurlyeq \delta,
\end{cases}
\end{equation}
and the transition rates for the LPP factorise as
\begin{equation*}
T_{\bdelta \bepsilon} = T^{\text{ul}}_{\delta\varepsilon} \prod_{(d,\lambda) \in \bdelta} \prod_{(e,\gamma) \in \bepsilon|_d} M (\lambda, \gamma),
\end{equation*}
compare Lemma~\ref{factorisation_marginal}.
Note that $\varSigma$ is a process of progressive refinement, which never returns to a state $\not \preccurlyeq $ the current state. This is due to the absence of coalescence events in the law of large numbers regime, which means that the ancestral recombination \emph{graph} is actually a \emph{tree}.

\begin{remark} \label{branchingRW}
The LPP can be interpreted as a multitype branching random walk (BRW) on $L$, with the types given by the subsets of $[n]$. The particles move according to the transition kernel $M$, and, as evident from the product structure of the transitions in Eq.~\eqref{Tulentries},  undergo independent branching that is the same at every location; each individual of type $d$ branches with probability $r_\varepsilon^d$ into $|\varepsilon|$ individuals of types $e_1^{},\ldots,e^{}_{|\varepsilon|}$, where $\varepsilon = \{e_1^{},\ldots,e^{}_{|\varepsilon|}\}$.
\end{remark}

\section{Limiting and quasi-limiting behaviour of the LPP}
\label{sec:asymptotics}
We assume now that $M$ is primitive (that is, irreducible and aperiodic), which guarantees the existence of and convergence to a unique stable stationary distribution \mbox{$q = \big (q(\alpha) \big )_{\alpha \in L} \in \RR^L$} such that
\begin{equation}\label{qqM}
q^{\trans} = q^{\trans} M,
\end{equation}
where $\trans$ denotes transpose.

We also assume that 
\begin{equation} \label{generalassumption}
\bigwedge \{ \delta \in \AS({[n]}) : r_\delta^{} > 0 \} = \pmin.
\end{equation}
That is, the coarsest common refinement of all partitions with positive recombination probability is the trivial partition $\pmin$ of ${[n]}$ into singletons. This is only a matter of technical convenience; otherwise, we could simply consider as  a single site any set of sites that are not separated by any partition $\delta$ with $r_\delta^{} > 0$. Note that Eq.~\eqref{generalassumption} implies that $\pmin$ is the unique absorbing state of the unlabelled partitioning process. We can now explicitly state the asymptotic behaviour of the MRE~\eqref{recomig_ode}.

\begin{theorem}\label{MREasymptotics}
Under the above assumptions, one has
\[
 \lim_{t \to \infty} \mu_t = \mu_\infty = \big ( \mu_\infty(\alpha) \big )_{\alpha \in L},
 \]
 where
 \begin{equation}\label{muinftyalpha_prod}
   \mu_\infty(\alpha) =  \bigotimes_{i=1}^n  \mu_\infty^{\{i\}}(\alpha)
\end{equation}
and
\begin{equation}\label{muinftyalpha_expl}
 \mu_\infty^{\{i\}}(\alpha) \defeq \sum_{\beta \in L} q(\beta) \mu_0^{\{i\}}(\beta)
\end{equation}
for all $\alpha \in L$. The convergence is geometric, i.e. there is a $\gamma \in (0,1)$ such that 
\begin{equation*}
\mu_t = \mu_\infty + \bigo{\gamma^t}
\end{equation*}
as $t \to \infty$, uniformly in $\mu_0$.
\end{theorem}

This is in line with \citet[Theorem 3.1]{Buerger09}, which states that the solution of the MRE~\eqref{recomig_ode} approaches (at a uniform geometric rate) the submanifold defined by spatial stationarity and linkage equilibrium. Spatial stationarity means that
\begin{equation*}
\mu(\alpha) = \sum_{\beta \in L} q(\beta)\mu(\beta)
\end{equation*}
with $q$ of \eqref{qqM}; and, under the assumption~\eqref{generalassumption}, linkage equilibrium  means that 
$\mu(\alpha)$ is the product of its one-dimensional marginals, as in Eq.~\eqref{muinftyalpha_prod}. However, like the explicit time evolution in Theorem~\ref{linearisation}, the explicit expression in Eq.~\eqref{muinftyalpha_expl} seems to be new.

In view of Theorem~\ref{duality}, this result is highly plausible: almost surely (at a uniform geometric rate), the partitioning process will enter its unique absorbing state where all blocks are singletons. In the sequel, independent migration processes will, for each block, converge to the unique stationary distribution $q$, again at a geometric rate and uniformly in the initial distribution. This behaviour is also clear in terms of the BRW picture. At some point, the type of each particle is a singleton, whence the particles stop branching and just keep performing independent random walks; see Remark~\ref{branchingRW}.

For the formal proof, note that the uniform convergence of the migration processes follows directly from the primitivity of $M$ via standard theory \citep[Appendix, Thm.~2.3]{Karlin_Taylor}. That the partitioning process enters its absorbing state at a uniform geometric rate is the content of the following lemma. \newpage
\begin{lemma} \label{geometricabsorption}
Let 
\begin{equation*}
\eta \defeq \max_{\delta \in \AS({[n]}) \setminus \{ \pmin \}} T^{\rm{ul}}_{\delta\delta} < 1
\end{equation*}
be the maximal sojourn probability of the unlabelled partitioning process and let 
\begin{equation*}
\tau \defeq \min \{ t \in \NN_0 : \varSigma_t = \pmin \}
\end{equation*}
be its time to absorption.
Then, uniformly in the initial distribution,
\begin{equation*}
\PP (\tau > t) = \cO\big ((\eta + \varepsilon)^t \big )
\end{equation*}
for any $\varepsilon > 0$ as $t \to \infty$.
\end{lemma}

\begin{proof}
Since the state space is finite and the partitioning process never returns to a state $\not \preccurlyeq$ the current state, this Markov chain may jump at most a finite number of times, say $m$ times, before it is absorbed in $\pmin$. Thus, for any fixed $\varepsilon > 0$,
\begin{equation*}
\begin{split}
\PP(\tau > t) &\leqslant \PP(\text{the chain has performed at most $m$ jumps up to time $t$}) \\
&\leqslant \sum_{j = 0}^m \binom{t}{j} (1 - \eta)^j \eta^{t - j} \\
&\leqslant \sum_{j = 0}^m \Big ( \frac{1 - \eta}{\eta}   \Big )^j t^m \eta^t = C' t^m \eta^t \leqslant C \eta^t \Big (\frac{\eta + \varepsilon}{\eta} \Big)^t = C(\eta + \varepsilon)^t,
\end{split}
\end{equation*}
where $C' = \sum_{j = 0}^m \Big ( \frac{1 - \eta}{\eta}   \Big )^j$ and  $C$ is sufficiently large. 
\end{proof}
  
Next, we investigate the asymptotic behaviour of the LPP.
\begin{prop}\label{LPPstationary}
There exists a $\gamma \in (0,1)$ such that 
\begin{equation*}
\PP \big ( \bSigma_t = \big \{(\{1\},\alpha_1),\ldots,(\{n\},\alpha_n)\big \} \big ) = \prod_{i  = 1}^n q(\alpha_i) + \bigo{\gamma^t}
\end{equation*}
as $t \to \infty$, uniformly in $\alpha^{}_1,\ldots,\alpha^{}_n \in L$ and the initial distribution of the \textnormal{LPP}.
For $\bdelta \in \LL \AS ({[n]})$ with $\delta \neq \pmin$, 
\begin{equation*}
\PP( \bSigma_t = \bdelta) = \cO \big ((\eta + \varepsilon)^t \big ),
\end{equation*}
for all $\varepsilon > 0$, again uniformly in the initial distribution.
\end{prop}

\begin{proof}
Let $\tau$ be as in Lemma~\ref{geometricabsorption}. The second statement follows immediately from Lemma~\ref{geometricabsorption} by noting that
\begin{equation*}
\PP(\bSigma_t = \bdelta) \leqslant\PP(\tau > t).
\end{equation*}
Now, assume that $\bdelta$ is of the form
\begin{equation*}
\bdelta = \big  \{ (\{1\},\alpha^{}_1),\ldots,(\{n\},\alpha^{}_n) \big \}.
\end{equation*}
Then, for all $\gamma_1 > \eta$,
\begin{equation}\label{splitsdfre}
\begin{split}
\PP( \bSigma_t = \bdelta)&\\ &= \PP \Big (\bSigma_t = \bdelta \mid \tau \leqslant \Big \lfloor \frac{t}{2} \Big \rfloor \Big )  \PP \Big (\tau \leqslant \Big \lfloor \frac{t}{2} \Big \rfloor \Big ) + 
 \PP \Big (\bSigma_t = \bdelta \mid \tau > \Big \lfloor \frac{t}{2} \Big \rfloor \Big ) \PP \Big (\tau > \Big \lfloor \frac{t}{2} \Big \rfloor \Big ) \\
&= \PP \Big (\bSigma_t = \bdelta \mid \tau \leqslant \Big \lfloor \frac{t}{2} \Big \rfloor \Big ) + \cO(\gamma_1^t)
\end{split}
\end{equation}
as $t \to \infty$, where the last step follows by an application of Lemma~\ref{geometricabsorption}. Furthermore,
\begin{equation}\label{migprocessesdqwefd}
\begin{split}
\PP \Big (\bSigma_t = \bdelta \mid \tau \leqslant \Big \lfloor \frac{t}{2} \Big \rfloor \Big ) &= \PP \Big (\Lambda_t^{(i)} = \alpha_i \text{ for all } 1 \leqslant i \leqslant n \mid \tau \leqslant \Big \lfloor \frac{t}{2} \Big \rfloor \Big ) \\ &= \prod_{i = 1}^n \PP \Big (\Lambda_t^{(i)} = \alpha_i \mid \tau  \leqslant \Big \lfloor \frac{t}{2} \Big \rfloor \Big ).
\end{split}
\end{equation}
Here, the $ \big ( \Lambda_t^{(i)} \big )_{t \in \NN_{\geqslant \tau}}$ for $i \in L$ are the labels of the (singleton) blocks from time $\tau$ onwards;
 they are   independent $L$-valued Markov chains with transition matrix $M$. By standard theory, we can be sure that, regardless of  the initial value, there is a $\gamma_2 \in (0,1)$ such that
\begin{equation*}
\PP \Big (\Lambda_t^{(i)} = \alpha_i \mid \tau  \leqslant \Big \lfloor \frac{t}{2} \Big \rfloor \Big ) = q(\alpha_i) + \bigo{\gamma_2^t},
\end{equation*}
uniformly in $\alpha_i$. 
Combining this with Eqs.~\eqref{splitsdfre} and \eqref{migprocessesdqwefd} proves the theorem. 
\end{proof}

\begin{proof}[Proof of Theorem~\textnormal{\ref{MREasymptotics}}]
By Theorem~\ref{duality}, Proposition~\ref{LPPstationary}, and Definition~\ref{labelledrecombinators}, we have for some $\gamma \in (0,1)$, independent of $\mu_0$, 
\begin{equation*}
\begin{split}
\mu_t(\alpha) &= \EE [\cR^{}_{\bSigma_t} (\mu_0^{}) \mid \bSigma_0 = \bpmax^\alpha ] \\
&= \sum_{\beta_1,\ldots,\beta_n \in L} \Big ( \prod_{i = 1}^n q(\beta_i) \Big )\EE [\cR^{}_{\bSigma_t} (\mu_0^{}) \mid \bSigma_0 = \bpmax^\alpha, \bSigma_t = \{ ( \{1\},\beta_1),\ldots,(\{n\},\beta_n) \} ] + \bigo{\gamma^t}\\
&= \sum_{\beta_1,\ldots,\beta_n \in L} \bigotimes_{i = 1}^n q (\beta_i)\mu_0^{\{i\}}(\beta_i) + \bigo{\gamma^t} \\
&= \bigotimes_{i = 1}^n \sum_{\beta \in L} q(\beta) \mu_0^{\{i\}}(\beta) + \bigo{\gamma^t} = \bigotimes_{i = 1}^n \mu_\infty^{\{i\}} (\alpha) + \bigo{\gamma^t} = \mu_\infty^{}(\alpha) + \bigo{\gamma^t}
\end{split}
\end{equation*}
\end{proof}

Since the asymptotic behaviour of the LPP is so simple, we now go one step further and inquire about its \emph{quasi}-limiting behaviour; that is, its asymptotic behaviour, conditioned on non-absorption of its base. Generally speaking, quasi-limiting distributions describe the first-order approximation of the deviation from the stationary behaviour. Recall that the partitioning process (labelled or unlabelled) is a process of progressive refinement, and never returns to a state coarser than the current state. This is very different from the situation considered by~\citet{cms}, where the focus is on irreducible chains.

%The biological significance of the quasi-limiting distribution is as follows. When investigating the joint effect of recombination together with other evolutionary forces, such as selection, their strength is often assumed to be small compared to recombination, so that the model can be seen as a perturbation of the pure recombination equation~\cite[Sec.~4.1]{Buerger09}. As such, its solution approaches a manifold that is close to the one defined by spatial homogeneity and linkage equilibrium; in the language of~\cite{Buerger09}, this behaviour is called \emph{quasi}-linkage equilibrium and its backward-time analogue is the quasi-limiting behaviour of the LPP. \todo{Das ist mir noch nicht so ganz klar...}

Unlike the limiting distribution, the quasi-limiting distribution will generally depend on the initial distribution. 
For convenience of notation, we let the LPP start from a maximal labelled partition $\bpmax^\alpha$, consisting of a single block with label $\alpha$. However, the following discussion can easily be adapted to the more general setting.
In what follows, we will exclude the pathological case of $r_{\pmin}^{} = 1$, where the probability of non-absorption is zero, and the conditional distribution we are interested in is not  well defined. 

We start by recalling the quasi-limiting behaviour of  $\varSigma$, which was already investigated by~\cite{Martinez}. We posit throughout that $\varSigma_0=\pmax$. To state the result, we need some additional notation. 
First, we define the set of states
\begin{equation*}
\AS^\downarrow({[n]}) \defeq \{ \delta \in \AS({[n]}) : \exists \ell \in \NN \text{ s.t. }\big ( (T^{\text{ul}})^\ell \big)_{\pmax \delta} >0 \}
\end{equation*}
that are \emph{reachable} by $\varSigma$ when starting in $\pmax$. As before, $\eta$ denotes the maximal sojourn probability of $\varSigma$ (compare Lemma~\ref{geometricabsorption}).
We will also need the set
\begin{equation*}
\cF \defeq \{\delta \in \AS^\downarrow({[n]}) : T^{\text{ul}}_{\delta\delta} = \eta\}
\end{equation*}
of reachable states with maximal sojourn probability. Note that our assumption $r_{\pmin}^{} \neq 1$ guarantees that $\eta > 0$. 
Finally, we define the \emph{first hitting time} of any given $\delta \in \AS({[n]})$,
\begin{equation*}
\tau_\delta^{} \defeq \min \{ t \in \NN_0 : \varSigma_t = \delta \},
\end{equation*} 
we write
$
\tau_\cF^{} \defeq \min_{\delta \in \cF} \tau_\delta^{}
$
for the first hitting time of $\cF$, and, as before, 
$
\tau = \tau_{\pmin}
$
for the time to absorption. 
The following result is known; see \citet[Theorem~5.5]{Martinez}.

\begin{theorem}\label{QLD_partitions}
For all $\delta \in \cF$, one has
\begin{equation*}
0 < \EE [ \eta^{- \tau_\delta^{}}; \tau_\delta^{} < \infty] \leqslant\EE [ \eta^{-\tau_{\cF}^{}}; \tau^{}_{\cF} < \infty ] < \infty.
\end{equation*}
For all $\delta \in \AS({[n]})$, the limit 
\begin{equation*}
\PP_{\textnormal{qlim}}^\varSigma (\delta) \defeq \lim_{t \to \infty} \PP(\varSigma_t = \delta \mid \tau > t)
\end{equation*}
exists and is given by
\begin{equation*}
\PP_{\textnormal{qlim}}^\varSigma (\delta) = \frac{\EE [ \eta^{- \tau_\delta^{}}; \tau_\delta^{} < \infty]}{\EE [ \eta^{-\tau_{\cF}^{}}; \tau^{}_{\cF} < \infty ]} \one_{\delta \in \cF}.
\end{equation*}
Thus defined, $\PP_{\textnormal{qlim}}^\varSigma$ is a probability measure on $\AS({[n]})$, called the \emph{quasi-limiting distribution} of $\varSigma$ 
$($starting from $\pmax$$)$.
\end{theorem} 

Recall that the labels of the different blocks evolve conditionally independently. Thus, we expect the quasi-limiting distribution of the LPP to be similar to the quasi-limiting distribution from Theorem~\ref{QLD_partitions}, garnished with the stationary distribution $q$ of the migration process. To be more explicit, we will prove the following result.

\begin{theorem} \label{QLD_migreco}
For all $\bdelta \in \LL \AS ({[n]})$,
\begin{equation*}
\lim_{t \to \infty} \PP(\bSigma_t = \bdelta \mid \tau > t) = \Big (\prod_{(d,\lambda) \in \bdelta} q(\lambda) \Big ) \PP_{\textnormal{qlim}}^\varSigma(\delta),
\end{equation*}
where $q$  is the unique stationary distribution \eqref{qqM} of the migration process.
\end{theorem}

\begin{remark}\label{QLDsignificance}
In Theorem~\ref{MREasymptotics}, we have approximated the solution of the MRE~\eqref{recomig_ode} by using Proposition~\ref{LPPstationary} to approximate the distribution of the labelled partitioning process by its limiting distribution. We can try to improve on this rather coarse estimate by also taking into account the quasi-limiting distribution; at least in principle, the disintegration
\begin{equation*}
\PP ( \bSigma_t = \bdelta) = \PP( \bSigma_t = \bdelta \mid \tau \leqslant t) \PP(\tau \leqslant t) + \PP( \bSigma_t = \bdelta \mid \tau >t ) \PP( \tau > t)
\end{equation*}
allows us to express the error term in Theorem~\ref{MREasymptotics} via the quasi-limiting distribution, at least when migration is strong compared to recombination. Acquiring precise asymptotics, however, would require more detailed knowledge about the probability $\PP(\tau > t)$ and the rate of convergence of the conditional distribution $\PP( \bSigma_t = \bdelta \mid \tau >t )$ to the quasi-limiting distribution.
\end{remark}

At the heart of the proof is the observation that further refinement of any $\delta \in \cF$ immediately leads to absorption; this was also one of the crucial ingredients in the proof of Theorem~\ref{QLD_partitions}, see \citet[Theorem 5.5]{Martinez} for the original reference\footnote{This result is also implicit in Lemma 5.1 of the corresponding corrigendum; for completeness, we present its proof.}.

\begin{lemma}\label{preabsorption}
For all $\delta \in \cF$, we have
\begin{equation}\label{secondlast}
T^{\rm{ul}}_{\delta\delta} + T^{\rm{ul}}_{\delta\pmin} = 1.
\end{equation}
\end{lemma}

\begin{proof}
We show that, for all $\delta \in \AS^\downarrow({[n]})$ with $T^{\rm{ul}}_{\delta\delta} + T^{\rm{ul}}_{\delta\pmin} \neq 1$, one has $\delta \notin \cF$. Indeed, for any such $\delta$, there is an $\varepsilon \notin \{\pmin, \delta\}$ with $T^{\rm{ul}}_{\delta\varepsilon} > 0$. For any such $\varepsilon$, there is at least one block $e \in \varepsilon$ with $\lvert e \rvert>1$. For any such $e$, the partition
\[
  \varepsilon' \defeq \{e\} \cup\big  \{ \{i \} : i \in {[n]} \setminus e \big \}  \prec \delta
\]
is reachable by Assumption~\eqref{generalassumption} (with ${[n]}$ replaced by individual blocks of $\delta$).  We then have
\[
T^{\text{ul}}_{\varepsilon'\varepsilon'} = r^e_{\{e\}} > \, r^{\tilde d}_{\{\tilde d\}} \prod_{\substack{d \in \delta \\ d \neq \tilde d, \lvert d \rvert>1}} r^d_{\{d\}}
=  \prod_{d \in \delta} r^d_{\{d\}} = T^{\text{ul}}_{\delta\delta}, 
\]
where $\tilde d$ is the block in $\delta$ that contains $e$.
The inequality is true since $\varepsilon' \prec \delta$ implies that either $\lvert e \rvert < \lvert \tilde d \rvert$, in which case $r^e_{\{e\}} > r^{\tilde d}_{\{\tilde d\}}$; or $\lvert \{d \in \delta: \lvert d \rvert > 1 \}\rvert > 1$, which entails that the constrained product is not empty (note that $r^d_{\{d\}}<1$ for $d$ with $\lvert d \rvert > 1$). We have thus proved that $\delta \notin \cF$.
\end{proof}

\begin{remark}
One might be tempted to assume that the sojourn probability is nondecreasing along every path
\begin{equation*}
\pmax \succcurlyeq \delta_1 \succcurlyeq \delta_2 \succcurlyeq \ldots \succcurlyeq \pmin
\end{equation*}
from the maximal partition to the absorbing state. To illustrate that this is not true in general, consider the following setup. Let $n = 4$ and assume the recombination distribution given by $r^{}_{\pmin} = \frac{1}{2}$, $r^{}_{\{\{1,2\},\{3,4\}\}} = \frac{1}{10}, r^{}_{\pmax} = \frac{2}{5}$ and $r^{}_\delta = 0$ otherwise. Then, the sojourn probability of the state $\pmax$ is $r^{}_{\pmax} = \frac{2}{5}$, while the (finer) state $\{\{1,2\},\{3,4\}\}$ has the smaller sojourn probability 
\[
r^{\{1,2\}}_{\{1,2\}} r^{\{3,4\}}_{\{3,4\}} = (1-r^{}_{\pmin})^2 = \frac{1}{4}.
\] 
\hfill $\diamondsuit$
\end{remark}

The idea of the proof of Theorem~\ref{QLD_migreco} is simple. First, notice that Lemma \ref{preabsorption} implies that conditional on non-absorption, 
$\varSigma$ remains constant after $\tau_{\cF}^{}$. From then on, the labels keep on evolving independently according to $M$, and their distributions converge to $q$. To make this rigorous, we just need to make sure that $t - \tau_\cF^{}$ is large enough (conditional on non-absorption). This is the content of the next Lemma.

\begin{lemma}\label{bounds}

\begin{enumerate}[label=\textnormal{(\alph*)}]
\item \label{lowerbound}
There exists $c > 0$ such that $\PP(\tau > t) \geqslant c \eta^t$ for all $t \in \NN$.
\item \label{upperbound}
Let $\eta' \defeq \max_{\delta \in \AS({[n]}) \setminus (\cF \cup \{\pmin\})} T^{\rm{ul}}_{\delta\delta}$. Then, for all $\eta'' > \eta'$, there exists $C > 0$ such that $\PP(\tau^{}_\cF \wedge \tau > t) \leqslant C (\eta'')^t$ for all $t \in \NN$.
\item \label{consequence}
There is a $\gamma \in (0,1)$ such that $\lim_{t \to \infty} \PP(\tau_{\cF}^{} > \gamma t \mid \tau > t) = 0$. 
\end{enumerate}
\end{lemma}

\begin{proof}
First, we show~\ref{lowerbound}. By definition, $\cF \subseteq \AS^\downarrow(I)$. Thus, there exists a $t_0 \in \NN$ such that $\PP(\tau_\cF^{} = t_0) > 0$. Then, we have for all $t \geqslant t_0$ that  
\begin{equation*}
\PP(\tau > t) \geqslant\PP(\tau > t, \tau_\cF^{} = t_0) = \PP(\tau > t \mid \tau_\cF^{} = t_0) \, \PP(\tau_\cF^{} = t_0) = c' \eta^{t - t_0} = (c' \eta^{-t_0} )\eta^t
\end{equation*}
with $c' = \PP(\tau_\cF^{} = t_0)$.
Note that we used Lemma~\ref{preabsorption} in the second-last step. Now, simply choose
\begin{equation*}
c \defeq \min \Bigl \{ \frac{\PP(\tau > t)}{\eta^t} :  0 \leqslant t \leqslant t_0 \Bigr \}\cup  \big \{c' \eta^{-t_0} \big \}.
\end{equation*}

For the proof of~\ref{upperbound}, we  couple $(\varSigma_t)_{t \in \NN_0}$ to another process $(N_t)_{t \in \NN_0}$ with values in \mbox{$\NN_0 \cup \{\infty\}$} and $N_0=0$. It evolves as follows. When $\varSigma_{t+1} =  \varSigma_t$, then $N_{t+1}\defeq N_t$ and when \mbox{$\varSigma_{t+1} \in \cF \cup \{ \pmin \}$}, we set $N_{t+1} \defeq \infty$. In all other cases, we perform a Bernoulli experiment with success probability
\begin{equation*}
\frac{1 - \eta'}{1 - T^{\text{ul}}_{\varSigma_t \varSigma_t}}.
\end{equation*} 
Upon success, we set $N_{t+1} \defeq N_t + 1$; otherwise, $N_{t+1} \defeq N_t$. 
Note that the marginal $(N_t)_{t \in \NN_0}$ of the coupling $(\varSigma_t,N_t)_{t \in \NN_0}$ stochastically dominates a process that has independent Bernoulli increments with parameter $1 - \eta'$. 

As we have argued before, the partitioning process can only jump a finite number of times before hitting either $\pmin$ or $\cF$. Thus, there is a positive integer $m$ such that, for all $t \in \NN$, $\tau \wedge \tau^{}_\cF > t$ implies $N_t \leqslant m$. Thus,
\begin{equation*}
\PP(\tau \wedge \tau_\cF^{} > t) \leqslant\PP(N_t \leqslant m) \leqslant \sum_{k = 0}^{m} \binom{t}{k} (1-\eta')^k (\eta')^{t - k} = P(t) (\eta')^t < C (\eta'')^t,
\end{equation*}
where $P(t)$ is a polynomial with degree $\leqslant m$, and $C$ and $\eta''$ are as stated.

Finally,~\ref{consequence} is a straightforward consequence of~\ref{lowerbound} and~\ref{upperbound}; first, fix any $\eta'' \in (\eta', \eta)$. Then, choose $\gamma$ such that $(\eta'')^\gamma < \eta$.
\end{proof}

After these preparations, the proof of Theorem~\ref{QLD_migreco} is not difficult.

\begin{proof}[Proof of Theorem~\textnormal{\ref{QLD_migreco}}]
Choose $\gamma$ as in~\ref{consequence} of Lemma~\ref{bounds}. We split
\begin{equation*}
\PP(\bSigma_t = \bdelta \mid \tau > t) = \PP(\bSigma_t = \bdelta, \tau_\cF^{} > \gamma t \mid \tau > t) + \PP(\bSigma_t = \bdelta, \tau_\cF^{} \leqslant \gamma t \mid \tau > t),
\end{equation*}
 The first probability tends to zero as $t \to \infty$, due to our choice of $\gamma$.  The second can be rewritten as
\begin{equation*}
\PP(\bSigma_t = \bdelta \mid \tau > t, \tau_\delta^{} \leqslant \gamma t)  \PP(\varSigma_t = \delta, \tau_\cF^{} \leqslant \gamma t \mid \tau > t),
\end{equation*}
where we have used that Lemma \ref{preabsorption} implies  $\{\tau>t,\tau^{}_\cF \leqslant \gamma t, \varSigma_t = \delta\}=\{\tau>t,\tau^{}_\delta \leqslant \gamma t\}$.
Here, the second factor converges to $\PP_{\text{qlim}}^\varSigma (\delta)$ by the choice of $\gamma$ and Lemma~\ref{bounds} \ref{consequence}.

Now consider the first factor. Together with $\tau > t$ and Lemma~\ref{preabsorption}, $\tau_\delta^{} \leqslant\gamma t$ implies  that $\varSigma_{s} = \delta$ for all $s$ between $\gamma  t$ and $t$.  
During this period, the labels of the blocks of $\delta$ evolve independently, and by the uniform convergence to the stationary distribution $q$, we obtain
\begin{equation*}
\lim_{t \to \infty} \PP(\bSigma_t = \bdelta \mid \tau > t, \tau_\delta^{} \leqslant \gamma t) = \prod_{(d,\lambda) \in \bdelta} q(\lambda),
\end{equation*} 
which completes the argument. For additional details, see also the proof of Proposition~\ref{LPPstationary}.
\end{proof}

\section{Recombination and migration in continuous time}
\label{sec:continuoustime}
Let us close by briefly discussing how our results carry over from the discrete-time to the continuous-time setting. To distinguish the notation from the discrete-time setting, we write $\omega = (\omega_t)_{t \geqslant 0}$ instead of $\mu = (\mu_t)_{t \in \NN_0}$. The recombination distribution $r$ is replaced by a collection $\varrho = (\varrho_\delta)_{\delta \in \SSS([n])}$ of non-negative recombination \emph{rates}; for each $\delta \in \AS([n])$, each individual is between time $t$ and $t + \dd t$ replaced by a new offspring that is recombined according to $\delta$,  with probability $\varrho_\delta^{} \dd t$.

Instead of the stochastic backward migration matrix, we use a Markov generator $N$ on $[n]$; between time $t$ and $t + \dd t$ and for $\alpha \neq \beta$, an individual at location $\alpha$ is, with probability $N(\alpha,\beta) \dd t$,  replaced by an individual from location $\beta$; we assume that this happens independently of recombination.

Putting this together means for the type distribution that we replace $\omega_t^{}(\alpha)$ by the convex combination
\begin{equation*}
\bigg (1 - \sum_{\delta \in \SSS([n])} \varrho_\delta^{} \dd t  - \sum_{\beta \neq \alpha} N(\alpha,\beta) \dd t \bigg ) \omega_t^{} + \sum_{\delta \in \SSS(I)} \varrho_\delta^{} \cR_\delta^{} (\omega_t^{}) \dd t + \sum_{\beta \neq \alpha} N(\alpha,\beta) \omega_t^{}(\beta) \dd t.
\end{equation*}
In other words,
\begin{equation}\label{continuousmigreco}
\dot{\omega}_t^{}(\alpha) = \sum_{\beta \in L} N(\alpha,\beta) \omega_t^{} (\beta) + \sum_{\delta \in \SSS([n])} \!  \varrho_\delta^{} \big ( \cR_\delta - \id) \omega_t^{}(\alpha);
\end{equation}
note that we have used that $N(\alpha,\alpha) = -\sum_{\beta \neq \alpha} N(\alpha, \beta)$ since $N$ is a Markov generator. The backward view can be easily adapted as follows. Again, we have an LPP (this time in continuous time) $\bSigmacont = (\bSigmacont_t)_{t \geqslant 0}$. It evolves as follows. At rate $\varrho_\varepsilon^{}$ for all $\varepsilon$, each labelled block $(d,\lambda)$ of $\bSigmacont_t$ is split into the blocks of the induced partition $\varepsilon|^{}_d$; each of these fragments inherits the label $\lambda$. In addition and independently, for every $\alpha \in L$, each block with label $\alpha$ is relabelled $\beta$ at rate $N(\alpha,\beta)$. Somewhat more formally, $\bSigmacont$ is a Markov chain in continuous time with generator $\bcQ$ defined by its nondiagonal elements
\begin{equation*}
\bcQ_{\bdelta \bepsilon} = \begin{cases}
\varrho_{\varepsilon|_d}^d, & \text{if } \bepsilon = ( \bdelta \setminus \{(d,\lambda)\} ) \cup \varepsilon|_d \times \{\lambda\} \text{ for some } d \in \delta, \\
N(\alpha,\beta), & \text{if } \bepsilon = ( \bdelta \setminus \{(d,\alpha)\} ) \cup \{(d,\beta)\} \text{ for some } d \in \delta, \\
0, & \text{otherwise},
\end{cases}
\end{equation*}
where the marginal recombination \emph{rates} are defined in analogy with the marginal recombination \emph{probabilities} (compare Eq.~\eqref{marginalrecombinationprobabilities}):
\begin{equation*}
\varrho^d_{\varepsilon} = \sum_{\substack{\varepsilon' \in \SSS([n]) \\ \varepsilon'|^{}_d = \varepsilon}} \varrho_{\varepsilon'}^{}.
\end{equation*}
Note that, in the case without migration and with recombination restricted to single crossovers, that is, to partitions of the form $\{[1:i],[i+1,n]\}$ for some $1 \leqslant i < n$, the continuous-time backward dynamics (and thus, by duality, the forward dynamics; see Eq.~\eqref{continuoustimeduality} below) has a simple explicit solution, which is due to the fact that crossover events ``rain down" on sequences in an independent Poissonian fashion~\citep{recoreview}. See also~\citet{LambertSchertzer} for the (much more involved) extension to the case with (a small amount of) coalescence in the infinite-sequence limit. 

But let us return to the full equation~\eqref{continuousmigreco}.
As before (compare Theorem~\ref{duality}), one can prove the duality relation
\begin{equation}\label{continuoustimeduality}
\cR_\bdelta^{} (\omega_t^{}) = \EE [\cR_{\bSigmacont_t}^{} (\omega_0^{}) \mid \bSigmacont_0 = \bdelta],
\end{equation}
whence we obtain the solution
\begin{equation}\label{matrixexponential}
\omega_t^{}(\alpha) = \sum_{\bdelta \in \LL \SSS([n])} (\ee^{t \bcQ})_{\bpmax^\alpha \bdelta} \cR_\bdelta^{} (\omega_0^{})
\end{equation}
by solving the associated (linear) Kolmogorov backward equation, in perfect analogy to Theorem~\ref{linearisation}. The duality relation can be proved by a straightforward adaptation of the techniques of~\citet{haldane}. Indeed, it was shown there that the recombination part of Eq.~\eqref{continuousmigreco} is dual to the splitting (or branching) part of $\bSigmacont$. Showing that the migration part is dual to the random walk defined by $N$ is a standard exercise.

Because Eq.~\eqref{matrixexponential} is not very concrete, let us derive a more explicit solution formula for the special case $n = 2$. We give a probabilistic argument, analogous to Eq.~\eqref{discreteexample}. First, note that with probability $\ee^{-\varrho^{}_{\pmin} t}$, the sites are not separated until time $t$, that is, $\Sigmacont_t = \pmax$; the single block has performed a random walk with transition kernel $N$ for the entire duration $t$. Hence, in this case, $\omega_t^{}(\alpha) = (\ee^{t N} \omega_0^{})(\alpha)$. On the other hand, if the blocks have been split at time $\sigma \in [0,t]$, then both sites have performed independent random walks, starting at time $\sigma$ at the location $\gamma$ where the split took place, and the solution is given by
$(\ee^{(t - \sigma) N} \omega^{}_0)^{\{1\}}(\gamma) \otimes (\ee^{(t - \sigma) N} \omega^{}_0)^{\{2\}}(\gamma)$.
Integrating over all possible values for $\sigma$ (keeping in mind that $\sigma$ is exponentially distributed with mean $1 / \varrho_{\pmin}^{}$) and $\gamma$ (keeping in mind that, at the moment of splitting, the block has label $\gamma$ with probability $(\ee^{\sigma N})(\alpha, \gamma)$), we obtain
\begin{equation} \label{continuousexample}
\omega_t^{}(\alpha) = \ee^{-\varrho_{\pmin}^{} t}  (\ee^{t N} \omega_0^{})(\alpha) + \varrho_{\pmin}^{} \sum_{\gamma \in L}  \int_0^t \ee^{-\varrho_{\pmin}^{} \sigma} (\ee^{\sigma N})(\alpha, \gamma)  (\ee^{(t - \sigma) N} \omega^{}_0)^{\{1\}}(\gamma) \otimes (\ee^{(t - \sigma) N} \omega^{}_0)^{\{2\}} (\gamma) \dd \sigma.
\end{equation}
For more than two loci, one can proceed in a similar fashion, disintegrating the solution conditional on the waiting time(s) between splitting events. However, this becomes cumbersome very quickly as one has to keep track of various different contributions, corresponding to different realisations of the jump chain of the (unlabelled) partitioning process. In particular, the form of these contributions changes, depending on the tree topology; see Fig.~\ref{topologies}.

\begin{figure}
\psfrag{a}{\hspace{-7mm}\raisebox{1mm}{$\{1,2,3\}$}}
\psfrag{c}{$\{2,3\}$}
\psfrag{e}{\raisebox{2.5mm}{$\{1\}$}}
\psfrag{f}{\hspace{-3mm}\raisebox{2.5mm}{$\{2\}$}}
\psfrag{g}{\hspace{-3mm}\raisebox{2.5mm}{$\{3\}$}}

\psfrag{a2}{\hspace{-7mm}\raisebox{1mm}{$\{1,2,3\}$}}
\psfrag{e2}{\raisebox{2.5mm}{$\{1\}$}}
\psfrag{f2}{\hspace{-3mm}\raisebox{2.5mm}{$\{2\}$}}
\psfrag{g2}{\hspace{-3mm}\raisebox{2.5mm}{$\{3\}$}}

\includegraphics[width = 0.7\textwidth]{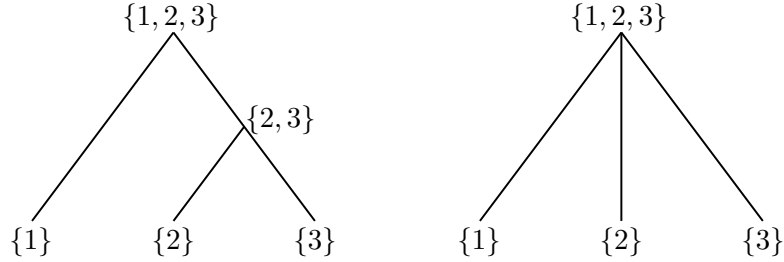}
\caption{\label{topologies}
Two realisations of the jump chain of the unlabelled partitioning process. Contributions to the solution corresponding to the left panel consist of two iterated integrals, while the topology on the right gives rise to a single integral as in Eq.~\eqref{continuousexample}. Note also that by permutation of the labels in the left panel, there are $2$ additional realisations of the jump chain with the same topology.
}
\end{figure}

\begin{remark}
It is straightforward to adapt the partitioning process in the diffusion limit to the setting with migration, both for finite, discrete sequences as treated by \citet{baakeesserprobst} as well as continuous sequences; see \citet{LambertSchertzer}. The new feature of this LPP with coalescence is that two blocks can only coalesce if they share the same label. However, an exhaustive treatment would go beyond the scope of this work.
\end{remark}

\section*{Acknowledgements}
It is our pleasure to thank Reinhard B\"urger for enlightening discussions and two anonymous referees for insightful comments.
This work was supported by the German Research Foundation (DFG), 
within the SPP 1590 (FA) and the CRC~1283, project C1 (EB); by  grant CMM Basal CONICYT, project AFB 170001 (SM); and
by ANID/Doctorado en el extranjero doctoral scholarship, grant number 2018-72190055 (IL).
Thanks for hospitality go from IL to the Bielefeld group and from EB to the Santiago group.

\bigskip

\bigskip

\end{document}